\documentclass[12pt]{article}
\usepackage{amsmath,amssymb}
\usepackage{amsthm}

\def\E{\mathbf{E}}

\def\P{\mathbf{P}}
\def\R{\mathbf{R}}

\def\1{\mathbf{1}}

\def\be{\beta}
\def\pa{\partial}
\def\ep{\epsilon}

\def\ga{\gamma}

\newtheorem{prop}{Proposition}[section]

\newtheorem{remark}{Remark}

\newcommand{\si}{\sigma}

\newcommand{\om}{\omega}

\newcommand{\Si}{\Sigma}

\begin{document}
\title{On a probabilistic derivation of the basic particle statistics
(Bose-Einstein, Fermi-Dirac, canonical, grand-canonical, intermediate) and related distributions}

\author{Vassili N. Kolokoltsov\thanks{Department of Statistics, University of Warwick,
 Coventry CV4 7AL UK,
  Email: v.kolokoltsov@warwick.ac.uk}}
\maketitle

\begin{abstract}
Combining intuitive probabilistic assumptions with the basic laws
of classical thermodynamics, using the latter to express probabilistic
parameters in terms of the thermodynamic quantities, we get a simple
unified derivation of the fundamental ensembles of statistical physics avoiding any
limiting procedures, quantum hypothesis and even statistical entropy maximization.
This point of view leads also to some related classes of correlated particle statistics.
\end{abstract}



\section{Introduction}

The Bose-Einstein (BE) and Fermi-Dirac (FD) statistics are
key concepts in modern physics, which remarkably start
penetrating even in social sciences (see e.g. \cite{Aert}).
There are many approaches to the derivation of the basic particle
distributions: canonical (or Boltzmann), grand canonical, BE and FD.
For instance, the canonical ensemble can be derived via equilibria
with some hypothetical external reservoirs or from the principle
of maximal entropy (see e.g. \cite{LaLi}). It can also be derived
from the microcanonical ensemble by passing to the limit of  projections
 to a single state of an ensemble of identical particles, as the number
of particles tends to infinity (see e. g. \cite{Khin}, \cite{Sca15}).
BE and FD distributions can be derived from the grand canonical ensemble
(which in turns is obtained via entropy optimization or via external
reservoirs) or via entropy optimization for certain energy level packing
models (see e.g. \cite{LaLi}). We can refer to \cite{Gorroochurn} for
the modern presentation of the original derivations due to Planck and
Bose. For the specific setting of the black body radiation the historical
background from the modern perspective can be found in \cite{Boyer}, see
also  a discussion in \cite{Persson}. Since the usual derivation of BE
statistics includes some law of large number limit, the pre-limit (number
of particle dependent) versions were introduced in \cite{Niven}. Usually
the BE distributions are associated with quantum behavior, and the
canonical Boltzmann distribution is obtained as their classical limit.

Papers \cite{CosGar97} and \cite{CosGar98} present pure probabilistic
derivations of the basic ensembles and the review of various approaches
to such derivations (including a curious idea of Brillouin assigning
particles positive or negative volume to derive Fermi-Dirac or Bose-Einstein
statistics respectively). These papers exploit the conditional probabilities
 arising from adding and taking away particles from an ensemble and derive
the distributions postulating certain properties of such probabilities.

In the present paper we follow a similar methodology, though searching
for the most direct postulate and combining it with the basic laws
of classical thermodynamics (coincidence of intensive variables
for systems in equilibria) in order to express probabilistic parameters
in terms of the basic thermodynamic quantities. Our conditioning postulate
is close in spirit to  Johnson's 'sufficientness postulate'. However, the latter
is given in terms of a Markov chain that 'creates' new particles
(as stressed in \cite{CosGar97}), and we employ a different point of view dealing
with a fixed finite collection of particles.
As a result we get a unified and very elementary derivation of all basic
distributions (including even more exotic intermediate statistics) avoiding
any limiting procedures, entropy maximization or quantum hypothesis.
Additionally this point of view leads to the derivation of more general
classes of particle distributions with correlated statistics.


In Section \ref{geomdist} we introduce our conditioning postulate revealing
a specific feature of the geometric distribution that relates it to particle
statistics. In the following sections we derive all basic statistics as
consequences of this feature and the classical laws of thermodynamics.
Finally we discuss some extensions showing, in particular, a remarkable
robustness of our conditioning postulate which leads to some interesting
 class of correlated statistics.

\section{Geometric distribution}
\label{geomdist}
Suppose one particle can be in one of $k$ states. The state space
of the system of many particles (in its statistical description)
consists of vectors $n=(n_1, \cdots, n_k)$ of $k$ non-negative
integers, where $n_j$ denotes the number of particles in the state
$j$. Adding a particle of type $j$ to such $n$ produces the new state
\[
n+e_j= (n_1, \cdots, n_{j-1}, n_j+1, n_{j+1}, \cdots, n_k),
\]
where $e_j$ is the unit coordinate vector
(with $j$th coordinate $1$ and other coordinates vanishing).

Let us denote by $n^+=(n_1, \cdots, n_k)^+$ the event that there
are at least $n_j$ particles in the state $j$ for each $j=1, \cdots, k$
(we found $n_j$ particles, but there can be more), that is
 \[
 (n_1, \cdots, n_k)^+
 =\cup_{m_1, \cdots, m_k: m_j \ge n_j \, \forall \, j} (m_1, \cdots, m_k).
 \]

Our 'conditioning postulate' is as follows: 
{\it the conditional probabilities} 
\[
q_j=\P((n+e_j)^+|n^+)
\]
\begin{equation}
\label{eqdefmanypartindi3}
=\P((n_1, \cdots, n_{j-1}, n_j+1, n_{j+1},
\cdots, n_k)^+ | (n_1, \cdots, n_{j-1}, n_j, n_{j+1}, \cdots, n_k)^+),
\end{equation}
{\it depend only on the type $j$ of a particle and not on the state} $n$.
This postulate is a kind of {\it no memory} property (it can be also interpreted
 as some no-interaction axiom).

\begin{prop}
Condition  \eqref{eqdefmanypartindi3} with some $q_j\in (0,1)$ is equivalent to the condition
\begin{equation}
\label{eqdefmanypartindi1}
\P(n_1, \cdots, n_{j-1}, n_j+1, n_{j+1}, \cdots, n_k)
=q_j \P(n_1, \cdots, n_{j-1}, n_j, n_{j+1}, \cdots, n_k),
\end{equation}
that is,  the ratio  $\P(n+e_j)/\P(n)=q_j$ depends only on the type of particles.
Each of conditions \eqref{eqdefmanypartindi3} and \eqref{eqdefmanypartindi1} is equivalent
to the formula
\begin{equation}
\label{eqdefmanypartindi5}
\P(n_1, \cdots, n_k)=\prod_{j=1}^k [q_j^{n_j}(1-q_j)],
\end{equation}
that is $(n_1,\cdots, n_k)$ is a random vector of independent geometrical distributions.
\end{prop}

\begin{proof}
 \eqref{eqdefmanypartindi1} $\Rightarrow$  \eqref{eqdefmanypartindi5}:
 Assume \eqref{eqdefmanypartindi1} holds with some $q_j\in (0,1)$.
Denoting by $P_0$ the probability of the vacuum state (the state without particles)
we find directly that
\begin{equation}
\label{eqdefmanypartindi2}
\P(n_1, \cdots, n_k)=P_0 \prod_{j=1}^k q_j^{n_j}.
\end{equation}

By the normalization condition for probabilities \eqref{eqdefmanypartindi2},
\[
1=P_0 \sum_{n_1=0}^{\infty}\cdots \sum_{n_k=0}^{\infty} \prod_{j=1}^k q_j^{n_j}
=\frac{P_0}{(1-q_1) \cdots (1-q_k)},
\]
implying \eqref{eqdefmanypartindi5}.

 \eqref{eqdefmanypartindi5} $\Rightarrow$  \eqref{eqdefmanypartindi3}:
It follows from \eqref{eqdefmanypartindi5} that
\begin{equation}
\label{eqdefmanypartindi30}
\P((n_1, \cdots, n_k)^+) =\sum_{m_1=n_1}^{\infty}\cdots \sum_{m_k=n_k}^{\infty} \prod_{j=1}^k \frac{q_j^{m_j}}{(1-q_j)}
=\prod_{j=1}^k q_j^{n_j},
\end{equation}
implying \eqref{eqdefmanypartindi3}.

\eqref{eqdefmanypartindi3} $\Rightarrow$  \eqref{eqdefmanypartindi5}.
Condition  \eqref{eqdefmanypartindi3} implies \eqref{eqdefmanypartindi30}
and hence the independence of all $n_j$. For $k=1$,
\eqref{eqdefmanypartindi5} follows directly.
\end{proof}

By the independence, the average number of particles in a state $j$ is independent
of other particles and equals the expectation of the corresponding geometric random variable
\begin{equation}
\label{eqdefmanypartindi6}
N_j=\E n_j=(1-q_j)\sum_{n_j=0}^{\infty} n_j q_j^{n_j}
= \frac{q_j}{1-q_j}
=\frac{1}{q_j^{-1}-1}.
\end{equation}
Inverting this formula we see that the values of $q_j$ can be uniquely identified
from the average number of particles in each state:
\begin{equation}
\label{eqdefmanypartindi6a}
q_j=\frac{1}{N_j^{-1}+1}.
\end{equation}

From physics one expects
\begin{equation}
\label{eqdefmanypartindi8}
 q_j=e^{\be(\mu-\ep_j)}, \quad \be =1/k_BT,
 \end{equation}
 where $\mu$ is the chemical potential, $T$ temperature,
in which case \eqref{eqdefmanypartindi6} concretizes to
\begin{equation}
\label{eqdefmanypartindi9}
\E n_j=\frac{1}{\exp \{\be(\ep_j-\mu)\}-1},
\end{equation}
which is the Bose-Einstein distribution.

In the next section we derive \eqref{eqdefmanypartindi8}-\eqref{eqdefmanypartindi9}
from \eqref{eqdefmanypartindi5} and basic thermodynamics.

\begin{remark} There are several well known and insightful ways to characterize
 geometric random vector  \eqref{eqdefmanypartindi5}. For instance, it can be derived
from the entropy maximization as the distribution on the collections $\{n_1, \cdots , n_k\}$
 maximizing the entropy  under the constraints of given $\E n_j$. Alternatively
 it arises as the monkey-typing process of Mandelbrot and Miller \cite{Miller}, which
 can be recast in terms of particle accumulation. It is also an
invariant distribution for the Markov chain on the collections $\{n_1, \cdots , n_k\}$ that
moves $n_j$ to $n_j+1$ or $n_j-1$ (the latter only when $n_j\neq 0$) with given probabilities
$r^j_+$ and $r^j_-$, in which case $q_j=r_+^j/r_-^j$ (used already in \cite{Champ},
see also \cite{YujiIj}). Yet another
way arises from packing randomly $k$ energy levels with indistinguishable particles,
each $j$th level having given number $L_j$
of states, so that given numbers $N_j$ of particles go to the
 $L_j$ states of the $j$th level (with all possible distribution equally probable). In this
way the probabilities $P_n$ to have $n$ particles in any given state of any $j$th level
can be described by the Yule-Simon growth process and become
geometric in the limit $N_j\to \infty$, $L_j\to \infty$ (see discussion and references in
\cite{GarSca} and \cite{SimRoy}). The limit in this scheme can be taken differently. Namely, let $L_j=a_jL$ with
fixed $a_j$ and let $L$ and $N=N_1+\cdots +N_k$ tend to infinity. As was noted in \cite{Mas05} the
limiting distribution depends on whether the ratio $L/N$ tends to infinity, to a finite number
or to zero, in which cases the limiting distribution is exponential (Gibbs canonical),
Bose-Einsten or power law (Pareto type) respectively.
\end{remark}

\section{Bose-Einstein distribution}

We shall now identify the expression for $q_j$ in terms of the classical variables
of thermodynamics, that is, obtain \eqref{eqdefmanypartindi8}. Of course some
properties of thermodynamic variables have to be taken into account for such a
derivation. We  shall use the principle that intensive variables, like temperature
and chemical potential, coincide for systems in equilibrium. This principle can be
taken for granted as an empirical fact or derived from the classical principle of
increasing entropy (that is, from the second law of thermodynamics).

System distributed by \eqref{eqdefmanypartindi5} can be looked at as $k$ systems
in equilibrium, each one characterized by its number of particles $N_j$, and having common
temperature $T$ and chemical potential $\mu$. Assume $\ep_j$ is the energy in state $j$.

Writing the fundamental equations for each system in terms of the grand potentials
$\Phi_j=E_j-S_jT-\mu N_j, \quad j=1, \cdots , k$,
the natural variables are $\mu,T$, and thus each $q_j$ must be a function of $\mu$ and $T$.
By \eqref{eqdefmanypartindi6}, the energy of the subsystem containing the states $j$ is
\[
E_j=\ep_j N_j=\ep_jq_j/(1-q_j).
\]
Consequently,
\[
\Phi_j=\frac{(\ep_j -\mu) q_j}{1-q_j}-k_BT [-\ln (1-q_j)-\frac{q_j}{1-q_j} \ln q_j],
\]
where $S_j=-\ln (1-q_j)-\frac{q_j}{1-q_j} \ln q_j$ is the entropy of the geometric distribution with the parameter $q_j$.
Since $\pa \Phi_j/\pa \mu =-N_j$ (by the definition of $\Phi$ as the Legendre transform of the energy $E=E(S,N)$) and  $N_j=q_j/(1-q_j)$,
\[
\frac{\pa \Phi_j}{\pa \mu}=-N_j+\frac{\pa \Phi_j}{\pa q_j}\frac{\pa q_j}{\pa \mu}=-N_j.
\]
Similarly, $\pa \Phi_j/\pa T =-S_j$, so that
\[
\frac{\pa \Phi_j}{\pa T}=-S_j+\frac{\pa \Phi_j}{\pa q_j}\frac{\pa q_j}{\pa T}=-S_j.
\]
As we cannot have both $\pa q_j/\pa \mu=0$ and $\pa q_j/\pa T=0$ (otherwise $q_j$ is a constant independent of
thermodynamics, so that the corresponding state $j$ can be considered as some irrelevant background) it follows that
\[
\frac{\pa \Phi_j}{\pa q_j}=\frac{\ep_j-\mu}{(1-q_j)^2}+k_BT \frac{\ln q_j}{(1-q_j)^2}=0,
\]
which implies \eqref{eqdefmanypartindi8} by expressing $q_j$ as the subject.

\section{Canonical and grand canonical ensembles}

Assuming \eqref{eqdefmanypartindi5}, what is the conditional probability of having a particle
in state $j$ given that there is only one particle in the system? It equals
\[
\P(j|\text{one particle})=\frac{P_0 q_j}{\sum_k P_0 q_k}.
\]
Under \eqref{eqdefmanypartindi8} it implies
\begin{equation}
\label{eqmanyparttocan1}
\P(j|\text{one particle})=Z^{-1} e^{-\be \ep_j}, \quad Z=\sum_l  e^{-\be \ep_l},
\end{equation}
that is, the standard canonical ensemble.

The grand canonical ensemble for bosons is just distribution \eqref{eqdefmanypartindi5}
with $q_j$ from \eqref{eqdefmanypartindi8}.
Conditioning on the total number of particles,  that is, taking the conditional probability
 \[
\P(n_1, \cdots , n_k|N)=\P(n_1, \cdots , n_k|\text{number of particles is} \, N),
\]
yields the canonical ensemble for $N$ particles.
By \eqref{eqdefmanypartindi8},
 \[
\P(n_1, \cdots , n_k|N)=\frac{\P(n_1, \cdots, n_k)}{\P (\text{number of particles is} \, N)}
= \frac{P_0 \prod e^{-\be (\ep_j-\mu)n_j}}{\P (\text{number of particles is} \, N)}
\]
that is
\begin{equation}
\label{eqmanyparttogrcan}
\P(n_1, \cdots , n_k|N)=Z_{gc}^{-1}(N) \prod e^{-\be (\ep_j-\mu)n_j},
\end{equation}
with
\begin{equation}
\label{eqmanyparttogrcan1}
Z_{gc}(N)=\frac{\P (\text{number of particles is} \, N)}{P_0}
=\sum_{n_1, \cdots , n_k:n_1+\cdots +n_k=N} \prod e^{-\be (\ep_j-\mu)n_j},
\end{equation}
the grand canonical partition function reduced to $N$ particle states.

One can calculate this function by induction yielding
\begin{equation}
\label{eqmanyparttogrcan1a}
Z_{gc}(N)=\sum_{j=1}^k q_j^{N+k-1}\prod_{m\neq j} (q_j-q_m)^{-1},
\end{equation}
in the case of different $q_j$ from \eqref{eqdefmanypartindi8}.

\section{Fermi-Dirac and intermediate statistics}

Similarly to the discussion above and keeping the main assumption \eqref{eqdefmanypartindi1},
 we can analyze the situation
 with the exclusion principle, that is, when the particles cannot occupy the same state, so that
the vector-states $(n_1, \cdots, n_k)$ can have coordinates zero or one only. In this case
\eqref{eqdefmanypartindi2} remains true, but only for this kind of vectors, and the normalization
 condition yields
\[
1=P_0 \sum_{n_1, n_2, \cdots, n_k=0}^1 \prod_{j=1}^k q_j^{n_j}=P_0(1+q_1) \cdots (1+q_k),
\]
so that
\begin{equation}
\label{eqdefmanypartindi10}
P_0=\prod_{j=1}^k (1+q_j)^{-1},
\quad \P(n_1, \cdots, n_k)=\prod_{j=1}^k \frac{q_j^{n_j}}{1+q_j}.
\end{equation}

Therefore the random vector $(n_1, \cdots, n_k)$ is an independent collection of $k$ Bernoulli
random variables, each taking values $0$ or $1$ with the probabilities $1/(1+q_j)$, $q_j/(1+q_j)$.

The average number of particles in state $j$ is thus the expectation of the $j$th Bernoulli random
 variable and equals
\begin{equation}
\label{eqdefmanypartindi12}
\E n_j=\frac{q_j}{1+q_j}
=\frac{1}{q_j^{-1}+1}.
\end{equation}

If $q_j$ are given by \eqref{eqdefmanypartindi8}, \eqref{eqdefmanypartindi12} turns to
\begin{equation}
\label{eqdefmanypartindi14}
\E n_j=\frac{1}{\exp \{\be(\ep_j-\mu)\}+1},
\end{equation}
which is the {\it Fermi-Dirac (FD) distribution}\index{Bose-Einstein distribution}.

Instead of \eqref{eqdefmanypartindi30},
for FD statistics one has
\[
\P((n_1, \cdots, n_k)^+) =\prod_{j:n_j=1} \frac{q_j}{1+q_j},
\]
as this is just the probability that all levels are occupied.

In a more general situation, the number of particles in each state can be bound by some
number $K\ge 1$. The corresponding {\it intermediate statistics} was initially suggested
by G. Gentile Jr (see \cite{CatBas} for a review). For instance, if one aims at using
Bose-Einstein statistics for molecules, their number is bounded (by the total number
of molecules), and the intermediate statistics with bounded occupation numbers may be
more realistic, than their $K\to \infty$ limit.

Under assumption \eqref{eqdefmanypartindi1} and limiting the occupation numbers by a
constant $K$, we get the normalization condition in the form
\[
1=P_0 \sum_{n_1, n_2, \cdots, n_k=0}^K \prod_{j=1}^k q_j^{n_j}
=P_0\prod_{m=1}^k \frac{1-q_m^{K+1}}{1-q_m},
\]
so that
\begin{equation}
\label{eqdefmanypartindi15}
\P(n_1, \cdots, n_k)=\prod_{j=1}^k \left(\frac{q_j^{n_j}(1-q_j)}{1-q_j^{K+1}}\right).
\end{equation}
Thus we obtain the {\it intermediate statistics} (sometimes also called {\it parastatistics})
 for the average number of particles in state $j$
(Gentile's formula) as another corollary of postulate \eqref{eqdefmanypartindi3}:
\begin{equation}
\label{eqdefmanypartindi16}
\E n_j=\frac{1-q_j}{1-q_j^{K+1}} \sum_{n=1}^K n q_j^n
=\frac{q_j[1+Kq_j^{K+1}-(1+K)q_j^K]}{(1-q_j)(1-q_j^{K+1})}.
\end{equation}
We refer to \cite{Mas12} for some recent applications of this statistics.

\section{Generalized Bose-Einstein distribution and canonical ensemble for magnetic systems}

The Bose-Einstein distribution \eqref{eqdefmanypartindi9} was derived from the geometric
distribution for the simplest system characterized only by the temperature and the chemical
properties of the energy levels. In general, different states of a system $\{1, \cdots, k\}$
 can be characterized by other local extensive variables, not only the energy $E$. Let us denote
them $U=(U_1,\cdots, U_m)$ and their normalized values (per particle) in $j$th state by
$u^j=(u_1^j,\cdots, u_m^j)$.
Denoting the dual intensive variables $\nu=(\nu_1, \cdots, \nu_m)$ we can write the thermodynamic
potential of the system with the basic variables $T,\nu$ as
\[
\Phi =E-ST-(\nu,U)=E-ST-\sum_{l=1}^m\nu_l U_l,
\]
and the corresponding thermodynamic potentials for subsystems combining particles in states $j$ as
\[
\Phi_j=\ep_j N_j-S_jT-(\nu, u^j)N_j
=\frac{(\ep_j -(\nu, u^j)) q_j}{1-q_j}-k_B T [-\ln (1-q_j)-\frac{q_j}{1-q_j} \ln q_j].
\]
Since $\pa \Phi_j/\pa \nu =-u_jN_j$ and $\pa \Phi_j/\pa T =-S_j$
(by the definition of the thermodynamic potential as the Legendre transform
of the energy $E=E(S, U)$), it follows that
\[
\frac{\pa \Phi_j}{\pa \nu}=-u_jN_j+\frac{\pa \Phi_j}{\pa q_j}\frac{\pa q_j}{\pa \nu}=-u_jN_j,
\]
and
\[
\frac{\pa \Phi_j}{\pa T}=-S_j+\frac{\pa \Phi_j}{\pa q_j}\frac{\pa q_j}{\pa T}=-S_j.
\]
As previously, we cannot have both $\pa q_j/\pa \nu=0$ and $\pa q_j/\pa T=0$, it follows that
\[
\frac{\pa \Phi_j}{\pa q_j}=\frac{\ep_j-(\nu, u^j)}{(1-q_j)^2}+k_BT \frac{\ln q_j}{(1-q_j)^2}=0,
\]
which implies the following extension of \eqref{eqdefmanypartindi8}:
\begin{equation}
\label{eqdefmanypartindi8new}
 q_j=e^{\be((\nu,u_j)-\ep_j)}.
 \end{equation}
 Formula \eqref{eqdefmanypartindi8} is obtained from \eqref{eqdefmanypartindi8new} if
 $m=1$, $u^j=1$ and $\mu=\nu$.

 The probability of a particle to be in $j$th state conditioned on having only one particle
 becomes now
  \begin{equation}
\label{eqmanyparttocan1new}
\P(j|\text{one particle})=Z^{-1} e^{-\be (\ep_j-(\nu,u_j))}, \quad Z=\sum_l  e^{-\be (\ep_l-(\nu,u_l))},
\end{equation}
extending \eqref{eqmanyparttocan1} and yielding the general version of the canonical ensemble.

 For instance, for the simplest magnetic system specified by a finite number of sites
 $\{1, \cdots, L\}$, each of which can have a spin $\si$ chosen from a fixed subset of
  a vector space (in the simplest case $\si=\pm 1$), a state is a configuration $\Si$, that
  is an assignment of $\si_l$, the values of $\si$ at each site $l$. A configuration $\Si$
  is characterized by its energy $E(\Si)$ (some given function) and the magnetization
  $M(\Si)=\sum \si_l$. The {\it canonical ensemble for such magnetic system} subject to an
  external magnetic field $H$ is the distribution on the configurations given by the formula
 \begin{equation}
\label{eqmanyparttocanmag}
\P(\Si)=Z^{-1} e^{-\be (E(\Si)-HM(\Si))}, \quad Z=\sum_{\Si}  e^{-\be (E(\Si)-HM(\Si))},
\end{equation}
(see e.g. \cite{LaLi}), which is seen to be given by \eqref{eqmanyparttocan1new} with the
index $j$ counting sites replaced by $\Si$, $\nu =H$ and $u_j$ denoted by $M(\Si)$.

Another example is the so-called {\it pressure ensemble} for gases obtained by choosing $\nu$
to be the pressure.

\section{Further links, extensions and exercises}

1. From \eqref{eqmanyparttogrcan1a} one can find the number of particles in state $i$
 conditioned on the total number $N$:
\[
\E (n_i|n_1+\cdots +n_k=N)=Z^{-1}_{gc}(N) \sum_{n_1+\cdots +n_k=N}
n_i \prod_{j=1}^k q_j^{n_j}
\]
\begin{equation}
\label{eqmanyparttogrcan1b}
=Z^{-1}_{gc}(N)\sum_{j\neq i}
\frac{q_iq_j^{k-2}[q_j^{N+1}+Nq_i^{N+1}-(N+1) q_i^N q_j]}{(q_j-q_i)^2 \prod_{m\neq i,j}(q_j-q_m)}.
\end{equation}
This formula is seen to be close to Gentile's intermediate distribution \eqref{eqdefmanypartindi16}.
In fact, \eqref{eqdefmanypartindi16} and \eqref{eqmanyparttogrcan1b} refer to the number of particles
under slightly different constraints.

What will be the limit of \eqref{eqmanyparttogrcan1b}, when $N\to \infty$? Suppose
\[
q_i >\max_{j\neq i} q_j.
\]
Then one can check (to perform calculations it is handy to start with $k=2$) that
\[
\lim_{N\to \infty} \E (n_i|n_1+\cdots +n_k=N)/N=1.
\]
The main point is this exact $1$ on the r.h.s., which means that almost
all particles will eventually  settle on the level $i$ of the lowest energy.
One can get even more precise result. Namely,
\begin{equation}
\label{eqmanyparttogrcan1c}
\lim_{N\to \infty} \E (n_j|n_1+\cdots +n_k=N)=\frac{q_j}{q_i-q_j}
=\frac{1}{e^{\be(\ep_i-\ep_j)}-1}, \quad j\neq i,
\end{equation}
that is, in the limit $N\to \infty$, other levels contain only finite number
 of particles, which are distributed according to the BE statistics on $k-1$ levels
 with the chemical potential coinciding with the lowest energy level.
This is a performance of the general effect of the Bose-Einstein condensation.

2. If in distributions \eqref{eqdefmanypartindi5} or \eqref{eqdefmanypartindi2},
all $q_j$ are close to each other, so that one can write $q_j=p+\ep_j$ with small
$\ep_j$, then, in the first order of approximation, \eqref{eqdefmanypartindi2} becomes
\[
\P(n_1, \cdots, n_k)=P_0 p^N (1+\frac{1}{p}\sum_j \ep_j n_j),
\]
with $N=\sum n_j$, that is, the r.h.s. is bilinear with respect to the
occupation numbers and transition rates. Such bilinear form is used by
the authors of \cite{Aert} in their psychological experiments  with 11 animals.

3. As was mentioned, condition \eqref{eqdefmanypartindi3} is reminiscent
to Johnson's  'sufficientness postulate' (see \cite{Zabel} for its full discussion)
stating that in the Markov process creating new particles the probability to create
a particle of type $i$ depends only on the number of existing particles of this type.
This is different from  \eqref{eqdefmanypartindi3} and leads one to a different
distribution. In particular, if this probability of creation
depends only on the type of a particle, the resulting probability of the occupation
numbers $n=(n_1, \cdots , n_k)$ becomes $N! \prod p_j^{n_j}/n_j!$ (see  \cite{CosGar97}),
which differs by the multinomial coefficient from the multivariate geometric.
We refer to \cite{Zabel97} for further extensions related to the
Johnson-Carnap continuum of inductive methods.

4. Let us now discuss a rather amazing robustness of our basic postulate
\eqref{eqdefmanypartindi3} and some correlated statistics arising from its extension.

\begin{prop}
\label{proponrobustgeom}
Assume that
\begin{equation}
\label{eqnewpost}
\P((n+e_j)^+|n^+)=q(j, \frac{n_j}{n_1+\cdots +n_k}), \quad n=(n_1, \cdots, n_k),
\end{equation}
that is, unlike our initial postulate, the conditional probabilities
on the l.h.s. of this equation are allowed to depend not only on $j$,
 but also on the fraction of $j$th particle in the state $n$.
 If $k>2$ it follows that
 \[
 q(j, \frac{m}{l})=q(j,1)=q_j
 \]
 for all $m/l \neq 0$, so that the deviation from $q_j$ not depending
on the fraction of $j$th particles in the state $n$ can actually manifest itself
only in the choice of $q_{0j}= q(j,0)$. For the unconditional probabilities
one gets the formula
\begin{equation}
\label{eqnewpost1}
\P((n_1, \cdots, n_k)^+)
=\om \prod_{j\in I} [q_{0j}q_j^{n_j-1}], \quad n_1+\cdots +n_k> 0,
\end{equation}
\begin{equation}
\label{eqnewpost2}
\P(n_1, \cdots, n_k)
= \prod_{j\in I} (1-q_j)\prod_{j\notin I}(1-q_{0j})
 \P((n_1, \cdots, n_k)^+), \quad n_1+\cdots +n_k> 0,
\end{equation}
where $I=\{j: n_j\neq 0\}$, and
\begin{equation}
\label{eqnewpost3}
\P(0, \cdots, 0)
=1-\om [1-\prod_{j=1}^k (1-q_{0j})].
\end{equation}
Here $q_j\in (0,1], q_{0j}\in (0,1], \om >0$ are arbitrary constants
subject to the constraint $\P(0, \cdots ,0)\ge 0$, that is
\begin{equation}
\label{eqnewpost4}
\om [1-\prod_{j=1}^k (1-q_{0j})]\le 1.
\end{equation}
In particular, $\om \le 1$ if $q_{0j}=1$ for at least one index $j$.
\end{prop}

The proof of this theorem is an insightful exercise based on the exploitation of the
consistency equations:
\begin{equation}
\label{eqnewpost5}
\P(n^+)= \P((n-e_j)^+)q(j, \frac{n_j-1}{n_1+\cdots +n_k}), \quad j=1, \cdots, k.
\end{equation}

If $\om=1$ in \eqref{eqnewpost1}, the vector $n=(n_1, \cdots, n_k)$ is seen to have independent
coordinates, which are represented by just slight extensions of the geometric distributions. However, if
$\om \neq 1$, the coordinates of vector $n$ become dependent:
\begin{equation}
\label{eqnewpost6}
\E n_j=\frac{\om q_{0j}}{1-q_j}, \quad \E (n_in_j)
= \frac{1}{\om}\E n_i \E n_j, \quad Cov (n_i,n_j)
=(\frac{1}{\om}-1)\E n_i \E n_j.
\end{equation}
Moreover, \eqref{eqnewpost1} is sensitive to the number of remaining types:
under the condition that $n_j=0$, the coefficient $\om$ for the remaining particles
turns to $\om(1-q_{0j})/(1-\om q_{0j})$.

The distribution of each coordinate $n_j$ is given by
\[
\om q_{0j}q_j^{n_j-1}(1-q_j), \quad n_j\neq 0,
\]
and the entropy of this distribution is found to be
\[
S_j=S_{Ber}(\om q_{0j})+\om q_{0j}S_{Geom}(q_j),
\]
where $S_{Ber}(a)=-a\ln a-(1-a)\ln (1-a)$ and $S_{Geom}(a)$ denote
the entropies of the Bernoulli and geometric random variables with
a parameter $a$. This allows one to find the difference between the
entropy of the vector $n$ and the sum of the entropies $S_j$ of $n_j$.
This difference vanishes if $\om=1$ (as it should be for independent
coordinates), and otherwise it represents the nontrivial entropy of
mixing of particles lying on different energy levels.

One can also check that for any $k>1$ the distribution \eqref{eqnewpost2}, \eqref{eqnewpost3}
can be obtained as the maximum entropy distribution on $\{0,1, \cdots \}^k$
subject to given expectations of the number of particles in each state, the
probabilities for each level to be nonempty and the probability of vacuum, that
is, by $2k+1$ parameters, which can be fixed by the choice of $q_j, q_{0j}, \om$.

It seems that for $k=2$ there are other distributions
satisfying \eqref{eqnewpost}, but it is not at all clear
(at least for the author), how they look like.


As shows already \eqref{eqmanyparttogrcan1c}, the general form of the grand
canonical distribution \eqref{eqdefmanypartindi5} is preserved under various
conditioning and limiting procedures. To support this claim one can also check,
for instance,  that under  \eqref{eqdefmanypartindi5}, the distribution
$\P((n_1,n_2)|n_3=n_1+n_2)$ has the same form with the parameters $k=2$, $q_1'=q_1q_3$,
$q_2'=q_2q_3$, and the distribution $\P((n_1,n_2)|n_1=n_2)$ again the same form
with the parameters $k=1$, $q'=q_1q_2$. However, examples of distribution
\eqref{eqnewpost2}, \eqref{eqnewpost3} can be also obtained from the standard grand
canonical distribution \eqref{eqdefmanypartindi5} by an appropriate conditioning,
for instance, by conditioning on the absence of vacuum. In fact, under \eqref{eqdefmanypartindi5},

\begin{equation}
\label{eqnewpostex}
\P((n_1, \cdots, n_k)|n_1+ \cdots +n_k>0)=\frac{\prod_{j=1}^k [q_j^{n_j}(1-q_j)]}{1-\prod_j(1-q_j)},
\end{equation}
which is \eqref{eqnewpost2}, \eqref{eqnewpost3} with $q_{0j}=q_j$ and $\om=[1-\prod_j(1-q_j)]^{-1}$.

5. A continuous variable version of axiom  \eqref{eqdefmanypartindi3} is the following condition
on the random vector $\tau=(\tau_1, \cdots, \tau_k)$ with non-negative coordinates:
\begin{equation}
\label{eqaxindexp}
\frac{\pa }{\pa s}|_{s=0} \P(\tau_j>t_j+s| \tau_l>t_l \,\, \forall l)=q_j.
\end{equation}
This is easily seen to imply that $\tau_j$ are independent exponential random variables.
The analog of \eqref{eqnewpost} is the condition
\begin{equation}
\label{eqaxindexp1}
\frac{\pa }{\pa s}|_{s=0} \P(\tau_j>t_j+s| \tau_l>t_l \,\, \forall l)
=q_j\left( \frac{t_j}{t_1+\cdots +t_k}\right).
\end{equation}

In analogy with Proposition \ref{proponrobustgeom} one can show that
if the vector has absolutely continuous distribution and $k>2$, then all $q_j$ on the r.h.s.
of \eqref{eqaxindexp1} must be constant, that is, \eqref{eqaxindexp} holds.
Some analogs of more general distributions \eqref{eqnewpost2} can be obtained
assuming the discontinuity of the distribution of $\tau$ on the boundary of its range.
For continuous random variables $\tau_j$ more natural interpretation is in terms
of time to default (finances) or survival time (engineering). Looking at
distributions \eqref{eqnewpost2}, \eqref{eqnewpost3} as natural discretizations
of continuous random vectors satisfying \eqref{eqaxindexp1} can lead to a performance
of the BE distributions for estimating the rates of defaults or survivals.

\end{document}